\documentclass[12pt]{amsart}

\usepackage{amssymb}
\usepackage{graphicx}
\usepackage{enumerate}
\usepackage{multirow}
\usepackage{amsmath,color}
\usepackage{hyperref}
\usepackage{url}
\usepackage{adjustbox}

\newtheorem{thm}{Theorem}[section]

\newtheorem{lem}[thm]{Lemma}
\newtheorem{prop}[thm]{Proposition}
\newtheorem{defn}[thm]{Definition}
\newtheorem{rem}[thm]{Remark}

\theoremstyle{definition}

\newtheorem{Problem}[thm]{Problem}

\makeatletter
\@namedef{subjclassname@2020}{%
  \textup{2020} Mathematics Subject Classification}
\makeatother

\title[On $Q$-polynomial distance-regular graphs with \ldots]{On $Q$-polynomial distance-regular graphs with a linear dependency involving a $3$-clique}

\author{Mojtaba Jazaeri}
\address{Department of Mathematics, Shahid Chamran University of Ahvaz, Ahvaz, Iran}
\email{M.Jazaeri@scu.ac.ir, M.Jazaeri@ipm.ir}
\begin{document}

\keywords{Distance-regular graph, $Q$-polynomial, classical parameters, regular near polygon.}
\subjclass{Primary: 05E30. Secondary: 05C25.}

\maketitle

\begin{abstract}
Let $\Gamma$ denote a distance-regular graph with diameter $D \geq 2$. Let $E$ denote a primitive idempotent of $\Gamma$ with respect to which $\Gamma$ is $Q$-polynomial. Assume that there exists a $3$-clique $\{x,y,z\}$ such that $E\hat{x},E\hat{y},E\hat{z}$ are linearly dependent. In this paper, we classify all the $Q$-polynomial distance-regular graphs $\Gamma$ with the above property. We describe these graphs from multiple points of view.
\end{abstract}

\section{Introduction}
This paper is about a certain kind of finite undirected graph, said to be distance-regular \cite[\S~4.1(A)]{BCN}, \cite[\S~2]{Terwilliger}. There is a well known property for a distance-regular graph, called the $Q$-polynomial property \cite[\S~4.1(E)]{BCN}, \cite[\S~11]{Terwilliger}. In this paper, we classify a certain type of $Q$-polynomial distance-regular graph. In our treatment, the following concepts will be relevant. We will consider the concepts of classical parameters \cite[\S~18]{Terwilliger}, negative type \cite{Weng-classical}, the cosine sequence \cite[\S~4]{Terwilliger}, and regular near $2D$-gons \cite[\S~6.4]{BCN}. Before we state our main results, we give some background about these concepts.

For the rest of this section, let $\Gamma$ denote a distance-regular graph with diameter $D \geq 2$. The concept of classical parameters was introduced in \cite[\S~6.1]{BCN}. If $\Gamma$ has classical parameters, then the intersection numbers of $\Gamma$ are given by attractive formulas in terms of four parameters $(D,b,\alpha,\sigma)$ \cite[\S~6.1(1a,1b)]{BCN}. See \cite[\S~6]{BCN}, \cite[\S~3.1.1]{DKT}, \cite[\S~18]{Terwilliger} for some results about classical parameters.

Assume that $\Gamma$ has classical parameters $(D,b,\alpha,\sigma)$. It is known that $b$ is an integer not equal to $0$ or $-1$ (cf. \cite[Proposition~6.2.1]{BCN}). The graph $\Gamma$ is said to have negative type whenever $b<-1$ \cite[\S~1]{Weng-classical}. Distance-regular graphs with negative type are investigated in \cite[\S~5.2]{DKT}, \cite{Weng-classical}, \cite{Weng}.

Associated with each eigenvalue of $\Gamma$, there is a sequence of scalars $\lbrace \sigma_{i} \rbrace_{i=0}^{D}$ called the cosine sequence. For $0 \leq i \leq D$, the scalar $\sigma_{i}$ can be interpreted as an angle cosine, see \eqref{equation cosine 2} below. It is known that the cosine sequence satisfies a $3$-term recurrence, see \cite[\S~14]{Terwilliger} and \eqref{sigma 1}, \eqref{equation cosine 1}, \eqref{equation cosine D} below. See \cite[\S~8.1]{BCN}, \cite[\S~2.5]{DKT}, \cite[\S~4]{Terwilliger} for some results about the cosine sequence.

A regular near $2D$-gon is a type of distance-regular graph with an attractive geometric structure \cite[\S~6.4]{BCN}. The regular near $2D$-gons were first introduced in \cite{ShY}. An example of the regular near $2D$-gons are the dual polar graphs \cite[\S~9.4]{BCN}. See \cite{DV} for a detailed study of the regular near $2D$-gons that are $Q$-polynomial. Other studies of the regular near $2D$-gons can be found in \cite[\S~6.6]{BCN}, \cite[\S~9.6]{DKT}, \cite{Terwilliger-Weng}.

Paul Terwilliger has stated the following problem (cf. \cite[Problem~1]{Terwilliger3}).
\begin{Problem} \label{Problem 1}
Let $\Gamma$ denote a distance-regular graph with diameter $D \geq 2$. Let $E$ denote a primitive idempotent of $\Gamma$ with respect to which $\Gamma$ is $Q$-polynomial. Assume that there exists a $3$-clique $\{x,y,z\}$ such that $E\hat{x},E\hat{y},E\hat{z}$ are linearly dependent. Investigate the combinatorial meaning of this condition.
\end{Problem}
In this paper, we investigate the $Q$-polynomial distance-regular graphs with the given property. The following is our main result.
\begin{thm} \label{main theorem}
Let $\Gamma$ denote a distance-regular graph with diameter $D \geq 2$. Let $E$ denote a primitive idempotent of $\Gamma$ with respect to which $\Gamma$ is $Q$-polynomial. Then the following are equivalent.
\begin{enumerate} [\rm(i)]
  \item There exists a $3$-clique $\{x,y,z\}$ such that $E\hat{x},E\hat{y},E\hat{z}$ are linearly dependent. \label{Condition 1}
  \item The graph $\Gamma$ has classical parameters $(D,b,\alpha,\sigma)=(D,-2,\alpha,2+\alpha-\alpha[_{1}^{D}])$ and $E$ is for the eigenvalue $\frac{b_{1}}{b}-1$.  \label{Condition 2}
  \item The graph $\Gamma$ is a regular near $2D$-gon of order $(2,t)$ and $E$ is for the eigenvalue $-t-1$.  \label{Condition 3}
  \item The intersection number $a_{1}=1$, and for every $3$-clique $\{x,y,z\}$ we have $E\hat{x}+E\hat{y}+E\hat{z}=0$. \label{Condition 4}
  \item The graph $\Gamma$ is one of those listed below, and $E$ is for the minimal eigenvalue of $\Gamma$. \label{Condition 5}
\begin{itemize}
  \item The unique regular near $4$-gon of order $(2,1)$,
  \item The unique regular near $4$-gon of order $(2,2)$,
  \item The unique regular near $6$-gon of order $(2,8)$,
  \item The unique regular near $6$-gon of order $(2,11)$,
  \item The unique regular near $6$-gon of order $(2,14)$,
  \item The dual polar graph $A_{2D-1}(2)$.
\end{itemize}
\item The cosine sequence $\{ \sigma_i \}_{i=0}^D$ for $E$ satisfies $\sigma_{i}=(-\frac{1}{2})^{i}$, where $0 \leq i \leq D$. \label{Condition 6}
\end{enumerate}
\end{thm}
This paper is organized as follows. In Section \ref{Pre}, we give some basic facts that will be used to state and prove our main results. In Section \ref{main section}, first we state and prove a sequence of lemmas and propositions. Next we use these results to prove Theorem \ref{main theorem}.
\section{Preliminaries} \label{Pre}
From now on, let $\Gamma$ be a connected graph with vertex set $X$ and diameter $D \geq 2$. For $x,y \in X$, let $d(x,y)$ denote the path-length distance between $x$ and $y$. Pick $x,y \in X$ and write $d(x,y)=i$. Let $b_{i}$ denote the number of neighbors of $x$ at distance $i+1$ from $y$, $a_{i}$ denote the number of neighbors of $x$ at distance $i$ from $y$, and $c_{i}$ denote the number of neighbors of $x$ at distance $i-1$ from $y$. The graph $\Gamma$ is called {\it distance-regular} whenever $a_{i}$, $b_{i}$, and $c_{i}$ are independent of $x,y$ and depend only on $i$. For the rest of this paper, assume that $\Gamma$ is distance-regular. Note that $\Gamma$ is regular with valency $k=b_{0}$, and that
\begin{equation} \label{Intersection array-degree}
k=a_{i}+b_{i}+c_{i} \;\;\;\; (0 \leq i\leq D),
\end{equation}
where $b_{D}=0$ and $c_{0}=0$. The sequence
\begin{equation*}
\{b_{0},b_{1},\ldots,b_{D-1}; c_{1},c_{2},\ldots,c_{D}\}
\end{equation*}
is called the {\it intersection array} of $\Gamma$. Pick $x \in X$. For $0 \leq i\leq D$, let $k_{i}$ denote the number of vertices in $X$ at distance $i$ from $x$. Note that $k_{0}=1$ and $k_{1}=k$. By a routine counting argument, we find $k_{i}c_{i}=k_{i-1}b_{i-1}$ for $1 \leq i\leq D$. It follows that $k_{i}$ is independent of choice of $x$.

Let $V=\mathbb{R}^{\vert X \vert}$ denote the vector space over $\mathbb{R}$, consisting of the column vectors with coordinates indexed by $X$ and all entries in $\mathbb{R}$. We endow $V$ with a bilinear form $\langle ~ , ~ \rangle$ that satisfies $\langle u , v \rangle= u^{t}v$ for $u,v \in V$, where $t$ denotes the transpose operator. We abbreviate $\langle u , u \rangle$ by $\|u\|^{2}$. We note that $\|u\|^{2} \geq 0$, with equality if and only if $u=0$. For $x \in X$, let $\hat{x}$ denote a vector in $V$ that has $x$-coordinate $1$ and all other coordinates $0$. Observe that the vectors $\{\hat{x} \mid x\in X\}$  form an orthonormal basis for $V$.

Let $Mat_{X}(\mathbb{R})$ denote the $\mathbb{R}$-algebra consisting of the matrices with rows and columns indexed by $X$ and all entries in $\mathbb{R}$. Let $A \in Mat_{X}(\mathbb{R})$ denote the adjacency matrix of $\Gamma$. Then the Bose-Mesner algebra of $\Gamma$ is the subalgebra of $Mat_{X}(\mathbb{R})$ generated by $A$. The Bose-Mesner algebra of $\Gamma$ has a basis $\{E_{i}\}_{i=0}^{D}$ such that $E_{0}=\vert X \vert^{-1}J$, $E_{i}E_{j}=\delta_{i,j}E_{i}$ ($0 \leq i,j \leq D$), and $\sum_{i=0}^{D} E_{i}=I$, where $I$ is the identity matrix and $J$ is the all-one matrix (cf. \cite[p.~4]{Terwilliger}). Following \cite{Terwilliger}, we call $\{E_{i}\}_{i=0}^{D}$ the {\it primitive idempotents} of $\Gamma$. The primitive idempotent $E_{0}$ is called {\it trivial}. For $B,C \in Mat_{X}(\mathbb{R})$, define the matrix $B \circ C \in Mat_{X}(\mathbb{R})$ with entries
\begin{equation*}
(B \circ C)_{y,z}=B_{y,z}C_{y,z} \;\;\;\; (y,z \in X).
\end{equation*}
The operation $\circ$ is called {\it entrywise multiplication}. Recall that the Bose-Mesner algebra of $\Gamma$ is closed under entrywise multiplication (see \cite[\S~5]{Terwilliger}). By \cite[Eqn.(8)]{Terwilliger}, there exist real numbers $q_{i,j}^{h}$ ($0\leq h,i,j \leq D$) such that
\begin{equation*}
E_{i} \circ E_{j}=\frac{1}{\vert X \vert}\sum_{h=0}^{D} q_{i,j}^{h}E_{h} \;\;\;\;\;(0 \leq i,j \leq D).
\end{equation*}
The parameters $q_{i,j}^{h}$ are called the {\it Krein parameters}. These parameters are nonnegative and this property is called {\it Krein condition} \cite[\S~5]{Terwilliger}. Because $\{E_{i}\}_{i=0}^{D}$ is a basis for the Bose-Mesner algebra of $\Gamma$, there exist real numbers $\{\theta_{i}\}_{i=0}^{D}$ such that
\begin{equation*}
A=\sum_{i=0}^{D}\theta_{i}E_{i}.
\end{equation*}
The scalars $\theta_i$ $(0 \leq i \leq D)$ are mutually distinct (see \cite[\S~4.1(B)]{BCN}). For $0 \leq i \leq D$, we have $AE_{i}=\theta_{i}E_{i}$. Therefore $\theta_{i}$ is an eigenvalue of $A$, and $E_{i}V$ is the corresponding eigenspace. By the {\it eigenvalues of $\Gamma$}, we mean the scalars $\theta_i$ $(0 \leq i \leq D)$. Let $m_{i}$ denote the dimension of $E_{i}V$. Then $\|E_{i}\hat{x}\|^{2}=\vert X \vert^{-1}m_{i}$ for all $x \in X$ (cf. \cite[Lemma~4.1(ii)]{Terwilliger}).

Let $E$ denote a primitive idempotent of $\Gamma$, and let $\theta$ denote the corresponding eigenvalue. We define a sequence of scalars $\lbrace \sigma_{i} \rbrace_{i=0}^{D}$ such that
\begin{align}
  & \sigma_{0}=1,\qquad \sigma_{1}= \frac{\theta}{k}, \label{sigma 1}\\
  & \theta \sigma_{i}=c_{i}\sigma_{i-1}+a_{i}\sigma_{i}+b_{i}\sigma_{i+1}\;\;\;\;1 \leq i \leq D-1. \label{equation cosine 1}
\end{align}
By \cite[Lemma~4.8]{Terwilliger}, we have
\begin{align}
\;\;c_{D}\sigma_{D-1}+a_{D}\sigma_{D}=\theta \sigma_{D}.&&&&\label{equation cosine D}
\end{align}
The sequence $\lbrace \sigma_{i} \rbrace_{i=0}^{D}$ is called the {\it cosine sequence} for $E$ (or $\theta$). This name is motivated by the following result. Pick $x,y \in X$ and write $d(x,y)=i$. By \cite[Lemma~4.1(iii)]{Terwilliger}, we have
\begin{equation} \label{equation cosine 2}
\sigma_{i}=\frac{\langle E\hat{x} , E\hat{y} \rangle }{\|E\hat{x}\|\|E\hat{y}\|}.
\end{equation}
Observe that $\sigma_{i}$ is the cosine of the angle between $E\hat{x}$ and $E\hat{y}$.

The graph $\Gamma$ is called {\it $Q$-polynomial} (with respect to the given ordering $\{E_{i}\}_{i=0}^{D}$ of the primitive idempotents) whenever the following holds for $0 \leq h,i,j \leq D$ (cf. \cite[Definition~11.1]{Terwilliger}):
\begin{enumerate}[\rm(i)]
\item $q_{i,j}^{h}=0$ if one of $h,i,j$ is greater than the sum of the other two,
\item $q_{i,j}^{h}\neq 0$ if one of $h,i,j$ is equal to the sum of the other two.
\end{enumerate}
In this case, we say that $\Gamma$ is {\it $Q$-polynomial with respect to $E$}, where $E=E_{1}$. Recall that if $\Gamma$ is $Q$-polynomial with respect to $E$, then the elements of $\lbrace \sigma_{i} \rbrace_{i=0}^{D}$ are mutually distinct (cf. \cite[Proposition~8.1.3]{BCN}).

Recall that our distance-regular graph $\Gamma$ is said to have {\it classical parameters $(D,b,\alpha,\sigma)$} whenever the intersection array satisfies
\begin{equation} \label{classical ci}
c_{i}=[^{i}_{1}](1+\alpha [^{i-1}_{1}]) \;\;\;\;\;(0 \leq i \leq D),
\end{equation}
\begin{equation} \label{classical bi}
b_{i}=([^{D}_{1}]-[^{i}_{1}])(\sigma-\alpha [^{i}_{1}]) \;\;\;\;\;(0 \leq i \leq D),
\end{equation}
where $[^{i}_{1}]=[^{i}_{1}]_{b}=1+b+\cdots+b^{i-1}$ for $1 \leq i \leq D$. Note that by a convention in \cite[\S~6.1(2)]{BCN}, we have $[^{0}_{1}]=0$. Recall that if $\Gamma$ has classical parameters, then $\Gamma$ is $Q$-polynomial with respect to the following ordering of the eigenvalues of $\Gamma$ (cf. \cite[Theorem~18.2, Lemma~18.3]{Terwilliger}):
\begin{equation}\label{eigenvalue sequence}
  \theta_{i}=\frac{b_{i}}{b^{i}}-[^{i}_{1}] \;\;\;\;\;(0 \leq i \leq D).
\end{equation}
\section{Main result} \label{main section}
In this section, we focus on Problem \ref{Problem 1}.
\begin{lem} \label{Lemma 1}
Let $\Gamma$ denote a distance-regular graph with diameter $D \geq 2$. Let $E$ denote a nontrivial primitive idempotent of $\Gamma$. Let $\{x,y,z\}$ denote a $3$-clique in $\Gamma$. Then the following items hold.
\begin{enumerate} [\rm(i)]
\item The matrix of inner products of $E\hat{x},E\hat{y},E\hat{z}$ is $\vert X \vert^{-1}mC$, where $m$ is the rank of $E$ and \label{inner product matrix}
\begin{equation*}
C=\left[
    \begin{array}{ccc}
      \sigma_{0} & \sigma_{1} & \sigma_{1} \\
      \sigma_{1} & \sigma_{0} & \sigma_{1} \\
      \sigma_{1} & \sigma_{1} & \sigma_{0} \\
    \end{array}
  \right].
\end{equation*}
\item The eigenvalues of $C$ are $1-\sigma_{1}$, $1-\sigma_{1}$, and $1+2\sigma_{1}$. \label{eigenvalue of C}
\item We have $1 > \sigma_{1} \geq -\frac{1}{2} $. \label{inequality for sigma1}
\item $\sigma_{1}=-\frac{1}{2}$ if and only if $E\hat{x},E\hat{y},E\hat{z}$ are linearly dependent. In this case, $E\hat{x}+E\hat{y}+E\hat{z}=0$.
\label{the value of sigma1}
\end{enumerate}
\end{lem}
\begin{proof}
\noindent \eqref{inner product matrix}: By \eqref{equation cosine 2}. \\
\noindent \eqref{eigenvalue of C}: Use linear algebra and recall that $\sigma_{0}=1$ by \eqref{sigma 1}. \\
\noindent \eqref{inequality for sigma1}: By construction $C$ is positive semidefinite, so its eigenvalues are nonnegative. We have $\sigma_{1} \neq 1$ since $E$ is nontrivial. \\
\noindent \eqref{the value of sigma1}: By linear algebra, we have $\sigma_{1}=-\frac{1}{2}$ if and only if $0$ is an eigenvalue of $C$ if and only if $C$ is singular if and only if $E\hat{x},E\hat{y},E\hat{z}$ are linearly dependent. Assume that this is the case. Then
\begin{equation*}
\| E\hat{x}+E\hat{y}+E\hat{z} \|^{2}=3\vert X \vert^{-1}m(1+2\sigma_{1})=0.
\end{equation*}
It follows that $E\hat{x}+E\hat{y}+E\hat{z}=0$.
\end{proof}
\begin{lem} \label{the lemmas}
Let $\Gamma$ denote a distance-regular graph with diameter $D \geq 2$. Let $E$ denote a primitive idempotent of $\Gamma$ with respect to which $\Gamma$ is $Q$-polynomial. Assume that there exists a $3$-clique $\{x,y,z\}$ such that $E\hat{x},E\hat{y},E\hat{z}$ are linearly dependent. Then the following items hold.
\begin{enumerate} [\rm(i)]
\item For $E$ the corresponding eigenvalue $\theta$ is equal to $-\frac{k}{2}$, and this is the minimal eigenvalue of $\Gamma$.
\label{minimal eigenvalue}
\item $E\hat{x}+E\hat{y}+E\hat{z}=0$ for all $3$-cliques $\{x,y,z\}$. \label{3-clique}
\item $\sigma_{i}=(-\frac{1}{2})^{i}$ for $0 \leq i \leq D$. \label{cosine sequence}
\item $a_{i}=c_{i}$ for $0 \leq i \leq D$. \label{parameters}
\item $a_{1}=1$. \label{a1}
\end{enumerate}
\end{lem}
\begin{proof}
\noindent \eqref{minimal eigenvalue}: We have $\theta=-\frac{k}{2}$ because $\theta=k \sigma_{1}$ and $\sigma_{1}=-\frac{1}{2}$ by Lemma \ref{Lemma 1}(\ref{the value of sigma1}). Moreover, $-\frac{k}{2}$ is the minimal eigenvalue of $\Gamma$ by Lemma \ref{Lemma 1}(\ref{inequality for sigma1}).\\
\noindent \eqref{3-clique}: Immediate from Lemma \ref{Lemma 1}(\ref{the value of sigma1}). \\
\noindent \eqref{cosine sequence}: We use induction on $i$. The result holds for $i=0$ since $\sigma_{0}=1$. The result holds for $i=1$ by \eqref{minimal eigenvalue} and \eqref{sigma 1}. For the rest of this proof, assume that $i \geq 2$. By induction, we assume that $\sigma_{j}=(-\frac{1}{2})^{j}$ for $0 \leq j \leq i-1$. We show that $\sigma_{i}=(-\frac{1}{2})^{i}$. We have $E\hat{x}+E\hat{y}+E\hat{z}=0$ by \eqref{3-clique}. Let $w \in X$ be a vertex at distance $i$ from $x$ and $i-1$ from $y$. Define $r=d(z,w)$. We show that $r=i$. By the triangle inequality, $r=i-1$ or $r=i$. Taking the inner product of $E\hat{w}$ and $E\hat{x}+E\hat{y}+E\hat{z}$, we find $\sigma_{i-1}+\sigma_{i}+\sigma_{r}=0$ by \eqref{equation cosine 2}. Suppose that $r=i-1$. Then $0=2\sigma_{i-1}+\sigma_{i}$. This implies that $\sigma_{i}=-2\sigma_{i-1}=\sigma_{i-2}$ by the induction hypothesis, which is a contradiction (the $\sigma_{i}$ are pairwise distinct as $\Gamma$ is $Q$-polynomial). Therefore $r=i$. We have $0=\sigma_{i-1}+2\sigma_{i}$. This implies that $\sigma_{i}=-\frac{1}{2}\sigma_{i-1}=(-\frac{1}{2})^{i}$ as desired.\\
\noindent \eqref{parameters}: If we substitute the data of \eqref{minimal eigenvalue} and \eqref{cosine sequence} in \eqref{equation cosine 1}, then we have
$k=4c_{i}-2a_{i}+b_{i}$, where $1 \leq i \leq D-1$. This implies that $a_{i}=c_{i}$ using \eqref{Intersection array-degree}. Moreover, we have $4c_{D}-2a_{D}=k$ by \eqref{equation cosine D} and the mentioned substitution. By this and $a_{D}+c_{D}=k$ we obtain $a_{D}=c_{D}$. \\
\noindent \eqref{a1}: Immediate from \eqref{parameters} since $c_{1}=1$.
\end{proof}
\begin{defn} \rm \label{Def. thick}
Let $\Gamma$ be a distance-regular graph with diameter $D \geq 2$. For $x \in X$, let $\Gamma(x)$ denote the set of neighbors of $x$ in $\Gamma$. The induced subgraph on $\Gamma(x)$ is called the {\it first subconstituent} or {\it local graph} of $\Gamma$ with respect to $x$. If $\Gamma(x)$ is a disjoint union of cliques, then each clique has size $a_{1}+1$ and there are $\frac{k}{a_{1}+1}$ such cliques. $\Gamma$ is said to be {\it locally} a {\it disjoint union of cliques} whenever $\Gamma(x)$ is a disjoint union of cliques for all $x \in X$. Assume that $\Gamma$ is locally a disjoint union of cliques. Then $\Gamma$ is said to have {\it order} $(s,t)$, where $s=a_{1}+1$ and $t+1=\frac{k}{a_{1}+1}$. If $s \geq 2$, then $\Gamma$ is called {\it thick} (cf. \cite[p.~35]{DKT}).
\end{defn}
\begin{defn} \rm
Let $\Gamma$ be a distance-regular graph with diameter $D \geq 2$. Then $\Gamma$ is called {\it a regular near $2D$-gon} whenever $\Gamma$ is locally a disjoint union of cliques and $a_{i}=a_{1}c_{i}$ for $1 \leq i\leq D$ (see \cite[p.~35]{DKT}).
\end{defn}
\begin{prop} \label{near polygon}
Let $\Gamma$ denote a distance-regular graph with diameter $D \geq 2$. Let $E$ denote a primitive idempotent of $\Gamma$ with respect to which $\Gamma$ is $Q$-polynomial. Assume that there exists a $3$-clique $\{x,y,z\}$ such that $E\hat{x},E\hat{y},E\hat{z}$ are linearly dependent. Then $\Gamma$ is a regular near $2D$-gon.
\end{prop}
\begin{proof}
We have $a_{1}=1$ by Lemma \ref{the lemmas}(\ref{a1}). Therefore $\Gamma$ is locally a disjoint union of $2$-cliques. Moreover, $a_{i}=c_{i}$ for $0 \leq i \leq D$ by Lemma \ref{the lemmas}(\ref{parameters}). This completes the proof.
\end{proof}
\begin{defn} (cf. \cite{Terwilliger2}) \rm \label{Kite}
Let $\Gamma$ be a distance-regular graph with diameter $D \geq 2$. For $2 \leq i \leq D$, a {\it kite of length $i$} is $4$-tuple $xyzw$ of vertices of $\Gamma$ such that $x,y,z$ are mutually adjacent and $w$ is at distance $d(x,w)=i$, $d(y,w)=i-1$, and $d(z,w)=i-1$.
\end{defn}
\begin{rem} \rm \label{Remark2}
Let $\Gamma$ be a distance-regular graph with diameter $D \geq 2$. If $\Gamma$ is kite-free, then it has no kite of length $2$ and therefore $\Gamma$ is locally a disjoint union of cliques. Let $\Gamma$ be a regular near $2D$-gon. Then every $3$-clique lies in a unique maximal clique in $\Gamma$. Furthermore, for a given $x \in X$ and maximal clique $C$ of $\Gamma$, there is a unique vertex $y \in C$ that is closest to $x$ (cf. \cite[\S~6.4]{BCN}). This implies that $\Gamma$ is kite-free. It follows that the distance-regular graph $\Gamma$ is a regular near $2D$-gon if and only if $\Gamma$ is kite-free and $a_{i}=a_{1}c_{i}$ for $1 \leq i\leq D$.
\end{rem}
\begin{lem} \label{classical parameters}
Let $\Gamma$ denote a distance-regular graph with diameter $D \geq 2$. Let $E$ denote a primitive idempotent of $\Gamma$ with respect to which $\Gamma$ is $Q$-polynomial. Assume that there exists a $3$-clique $\{x,y,z\}$ such that $E\hat{x},E\hat{y},E\hat{z}$ are linearly dependent. Then $\Gamma$ has classical parameters $(D,b,\alpha,\sigma)$, where
\begin{equation*}
b=-2, \qquad \qquad \sigma=2+\alpha-\alpha[_{1}^{D}].
\end{equation*}
Moreover, $\alpha=-1-c_{2}$.
\end{lem}
\begin{proof}
Using Lemma \ref{the lemmas}(\ref{cosine sequence}) and from \cite[Theorem~8.4.1]{BCN}, $\Gamma$ has classical parameters $(D,b,\alpha,\sigma)$ with $b=-2$. This implies that $\alpha=-1-c_{2}$ by \eqref{classical ci}. We have $b_{0}=k$ and therefore $k=\sigma[_{1}^{D}]$ by \eqref{classical bi}. Moreover, $b_{1}=k-2$ by Lemma \ref{the lemmas}(\ref{a1}). By substituting $k=\sigma[_{1}^{D}]$ in $b_{1}=k-2$ and using \eqref{classical bi}, we have $\sigma=2+\alpha-\alpha[_{1}^{D}]$. This completes the proof.
\end{proof}
\begin{lem} \label{Restriction on c2}
Let $\Gamma$ denote a distance-regular graph with diameter $D \geq 2$. Let $E$ denote a primitive idempotent of $\Gamma$ with respect to which $\Gamma$ is $Q$-polynomial. Assume that there exists a $3$-clique $\{x,y,z\}$ such that $E\hat{x},E\hat{y},E\hat{z}$ are linearly dependent. Then $1 \leq c_{2} \leq 5$.
\end{lem}
\begin{proof}
Since $\Gamma$ is distance-regular with diameter $D \geq 2$, we have $c_{2} \geq 1$. We show that $c_{2} \leq 5$. Pick $x,y \in X$ with $d(x,y)=2$. Note that $\vert \Gamma(x) \cap \Gamma(y) \vert = c_{2}$. Also note that two distinct vertices in $\Gamma(x) \cap \Gamma(y)$ are at distance $2$, because $a_{1}=1$ by Lemma \ref{the lemmas}(\ref{a1}). Define
\begin{equation*}
u=E\hat{x}+E\hat{y},
\end{equation*}
and
\begin{equation*}
v=\sum _{z \in \Gamma(x) \cap \Gamma(y)}E\hat{z}.
\end{equation*}
By the Cauchy-Schwarz inequality,
\begin{equation} \label{Cauchy-Schwarz}
\langle u,v \rangle^{2} \leq \langle u,u \rangle  \langle v,v \rangle.
\end{equation}
Using the data in \eqref{equation cosine 2} and Lemma \ref{the lemmas}, we obtain
\begin{equation} \label{equation 1-1}
\langle u,v \rangle=-c_{2}m\vert X \vert^{-1},
\end{equation}
\begin{equation} \label{equation 2-1}
\langle u,u \rangle= \frac{5m}{2}\vert X \vert^{-1},
\end{equation}
\begin{equation} \label{equation 3-1}
\langle v,v \rangle=\frac{(c_{2}^{2}+3c_{2})m}{4}\vert X \vert^{-1},
\end{equation}
where $m$ denotes the rank of $E$. Evaluating \eqref{Cauchy-Schwarz} using \eqref{equation 1-1}, \eqref{equation 2-1}, and \eqref{equation 3-1}, we obtain  $c_{2}(5-c_{2}) \geq 0$. By this, we have $c_{2} \leq 5$. This completes the proof.
\end{proof}
\begin{prop} \label{D=2}
Let $\Gamma$ denote a distance-regular graph with $D=2$. Let $E$ denote a primitive idempotent of $\Gamma$ with respect to which $\Gamma$ is $Q$-polynomial. Assume that there exists a $3$-clique $\{x,y,z\}$ such that $E\hat{x},E\hat{y},E\hat{z}$ are linearly dependent. Then $\Gamma$ is isomorphic to one of the following graphs.
\begin{itemize}
\item The unique regular near $4$-gon of order $(2,1)$,
\item The unique regular near $4$-gon of order $(2,2)$,
\item The unique regular near $4$-gon of order $(2,4)$.
\end{itemize}
\end{prop}
\begin{proof}
The graph $\Gamma$ is a regular near $4$-gon by Proposition \ref{near polygon}. Moreover, $a_{1}=1$ by Lemma \ref{the lemmas}(\ref{a1}). The regular near $4$-gons with $a_{1}=1$ are classified in \cite[p.~30(Examples)]{BCN}. The result follows from that classification.
\end{proof}
\begin{prop} \label{diameter 3}
Let $\Gamma$ denote a distance-regular graph with diameter $D = 3$. Let $E$ denote a primitive idempotent of $\Gamma$ with respect to which $\Gamma$ is $Q$-polynomial. Assume that there exists a $3$-clique $\{x,y,z\}$ such that $E\hat{x},E\hat{y},E\hat{z}$ are linearly dependent. Then $\Gamma$ is isomorphic to one of the following graphs.
\begin{itemize}
  \item The unique regular near $6$-gon of order $(2,8)$,
  \item The unique regular near $6$-gon of order $(2,11)$,
  \item The unique regular near $6$-gon of order $(2,14)$,
  \item The dual polar graph $A_{5}(2)$.
\end{itemize}
\end{prop}
\begin{proof}
By Lemma \ref{Restriction on c2}, we have $1 \leq c_{2} \leq 5$. For each choice of $c_{2}$, we compute the intersection array using Lemma \ref{classical parameters} and \eqref{classical ci},\eqref{classical bi}. The results are in the following table.
\begin{center}
\begin{tabular}{|c|c|}
  \hline
  $c_{2}$ & \rm Intersection array \\
  \hline
  $1$ & $\{18,16,16;1,1,9\}$ \\
  $2$ & $\{24,22,20;1,2,12\}$ \\
  $3$ & $\{30,28,24;1,3,15\}$ \\
  $4$ & $\{36,34,28;1,4,18\}$ \\
  $5$ & $\{42,40,32;1,5,21\}$ \\
  \hline
\end{tabular}
\end{center}
Assume that $c_{2}=1$. Then $\Gamma$ exists and is unique by \cite[p.~427]{BCN}. Assume that $c_{2}=2$. Then $\Gamma$ exists and is unique by \cite[p.~427]{BCN}. Assume that $c_{2}=3$. Then $\Gamma$ exists and is unique by \cite[p.~428]{BCN}. Assume that $c_{2}=4$. Then $\Gamma$ does not exist. Indeed the intersection array is not feasible for the following reason. By Lemma \ref{the lemmas}(\ref{minimal eigenvalue}), $\theta=-18$ is an eigenvalue of $\Gamma$ and by the Bigg's formula \cite[Theorem~2.8]{DKT} the multiplicity of $\theta$ is not integer. In fact, using Lemma \ref{the lemmas}(\ref{cosine sequence}) and the intersection array of $\Gamma$, the multiplicity of $\theta$ is
\begin{equation*}
\frac{\sum_{i=0}^{3}k_{i}}{\sum_{i=0}^{3}k_{i}\sigma_{i}^{2}}=
\end{equation*}
\begin{equation*}
\frac{819}{1+36(\frac{1}{4})+306(\frac{1}{16})+476(\frac{1}{64})}=22.4.
\end{equation*}
Assume that $c_{2}=5$. Then $\Gamma$ exists and is unique by \cite[p.~428]{BCN}. This completes the proof.
\end{proof}
\begin{prop} \label{c2=1}
Let $\Gamma$ denote a distance-regular graph with diameter $D \geq 4$. Let $E$ denote a primitive idempotent of $\Gamma$ with respect to which $\Gamma$ is $Q$-polynomial. Assume that there exists a $3$-clique $\{x,y,z\}$ such that $E\hat{x},E\hat{y},E\hat{z}$ are linearly dependent, and $c_{2}=1$. Then $\Gamma$ does not exist.
\end{prop}
\begin{proof}
By Lemma \ref{classical parameters} with $c_{2}=1$, we find that $\Gamma$ has classical parameters $(D,-2,-2,2[_{1}^{D}])$. The graph $\Gamma$ does not exist by \cite[Corollary~5.4]{DV}.
\end{proof}
\begin{prop} \label{D>=4}
Let $\Gamma$ denote a distance-regular graph with diameter $D \geq 4$ and $c_{2} \geq 2$. Let $E$ denote a primitive idempotent of $\Gamma$ with respect to which $\Gamma$ is $Q$-polynomial. Assume that there exists a $3$-clique $\{x,y,z\}$ such that $E\hat{x},E\hat{y},E\hat{z}$ are linearly dependent. Then $c_{2}=5$ and $\Gamma$ is the dual polar graph $A_{2D-1}(2)$.
\end{prop}
\begin{proof}
The graph $\Gamma$ has classical parameters $(D,b,\alpha,\sigma)$, where $b=-2$, by Lemma \ref{classical parameters}. Therefore $\Gamma$ is the dual polar graph $A_{2D-1}(2)$ by \cite[Theorem~B]{Weng} because $a_{1}=1$ by Lemma \ref{the lemmas}(\ref{a1}). Moreover, $c_{2}=5$ by Lemma \ref{classical parameters} because $\alpha=-6$ by \cite[Tbl.~1]{DKT}. This completes the proof.
\end{proof}
Now we can prove Theorem \ref{main theorem}.

\noindent {\it Proof of Theorem \ref{main theorem}.} \eqref{Condition 1} $\Rightarrow$ \eqref{Condition 2}: By Lemma \ref{classical parameters}, $\Gamma$ has classical parameters $(D,b,\alpha,\sigma)=(D,-2,\alpha,2+\alpha-\alpha[_{1}^{D}])$. Moreover, the corresponding eigenvalue for $E$ is equal to $-\frac{k}{2}$ by Lemma \ref{the lemmas}(\ref{minimal eigenvalue}). We have $-\frac{k}{2}=\frac{b_{1}}{b}-1$ by Lemma \ref{the lemmas}(\ref{a1}) and $b=-2$. Therefore, $E$ is for the eigenvalue $\frac{b_{1}}{b}-1$. \\
\noindent \eqref{Condition 2} $\Rightarrow$ \eqref{Condition 3}: Using \eqref{Intersection array-degree} along with \eqref{classical ci}, \eqref{classical bi}, we find that $a_{i}=c_{i}$ for $1 \leq i \leq D$. In particular $a_{1}=c_{1}=1$, so $\Gamma$ has no $2$-kites. By these comments and Remark \ref{Remark2}, we see that $\Gamma$ is a regular near $2D$-gon. With reference to Definition \ref{Def. thick}, we see that $\Gamma$ has order $(s,t)$, where $s=a_{1}+1=2$ and $t=\frac{k}{a_{1}+1}-1=\frac{k}{2}-1$. Note that $\frac{b_{1}}{b}-1=-\frac{k}{2}=-t-1$. The result follows. \\
\noindent \eqref{Condition 3} $\Rightarrow$ \eqref{Condition 4}: The intersection number $a_{1}=1$ because $\Gamma$ is a regular near $2D$-gon of order $(2,t)$. Moreover, the corresponding eigenvalue for $E$ is equal to $-\frac{k}{2}$ since $t=\frac{k}{2}-1$. This implies that $\sigma_{1}=-\frac{1}{2}$ by \eqref{sigma 1}, and the result follows by Lemma \ref{Lemma 1}(\ref{the value of sigma1}) and Lemma \ref{the lemmas}(\ref{3-clique}). \\
\noindent \eqref{Condition 4} $\Rightarrow$ \eqref{Condition 5}: We refer to the table in Remark \ref{Remark}. First assume that $D=2$. Then, $c_{2}=2,3,5$ and $\Gamma$ is the unique regular near $4$-gon of order $(2,t)$, where $t=1,2,4$, by Proposition \ref{D=2}. Moreover, the intersection array of the unique regular near $4$-gon of order $(2,4)$ is the same as intersection array of the dual polar graph $A_{3}(2)$ (cf. \cite[Theorem~9.4.3]{BCN}). Next assume that $D=3$. Then $c_{2}=1,2,3,5$ and $\Gamma$ is the unique regular near $6$-gon of order $(2,t)$, where $t=8,11,14$, or the dual polar graph $A_{5}(2)$ by Proposition \ref{diameter 3}. Next assume that $D \geq 4$. Then by Proposition \ref{c2=1} and Proposition \ref{D>=4}, $c_{2}=5$ and $\Gamma$ is the dual polar graph $A_{2D-1}(2)$. \\
\noindent \eqref{Condition 5} $\Rightarrow$ \eqref{Condition 6}: It is easily checked that for each of the cases listed in \eqref{Condition 5}, the cosine sequence of $E$ satisfies $\sigma_{i}=(-\frac{1}{2})^{i}$ for $0 \leq i \leq D$. \\
\noindent \eqref{Condition 6} $\Rightarrow$ \eqref{Condition 1}: We have $-\frac{k}{2}\sigma_{1}=\sigma_{0}+a_{1}\sigma_{1}+(k-a_{1}-1)\sigma_{2}$ by \eqref{equation cosine 1} and using \eqref{Intersection array-degree}. This implies that $\frac{k}{4}=1-\frac{a_{1}}{2}+\frac{k-a_{1}-1}{4}$ and therefore $a_{1}=1$. Thus the result follows by Lemma \ref{Lemma 1}(\ref{the value of sigma1}). This completes the proof.
\begin{rem} \rm \label{Remark}
In the following tables, we bring out some properties of the distance-regular graphs listed in item \eqref{Condition 5} of Theorem \ref{main theorem}.
\begin{center}
\begin{adjustbox}{width=1\textwidth}
\begin{tabular}{|l|c|c|c|}
  \hline
  Name of graph& $D$ & $\{b_{i}\}_{i=0}^{D-1}$& $\{c_{i}\}_{i=1}^{D}$ \\
  \hline
  Regular near $4$-gon of order $(2,1)$ & $2$ & $4,2$ & $1,2$ \\
  Regular near $4$-gon of order $(2,2)$ & $2$ & $6,4$ & $1,3$\\
  Regular near $6$-gon of order $(2,8)$ & $3$ & $18,16,16$ & $1,1,9$\\
  Regular near $6$-gon of order $(2,11)$ & $3$ & $24,22,20$ & $1,2,12$\\
  Regular near $6$-gon of order $(2,14)$ & $3$ & $30,28,24$ & $1,3,15$\\
  Dual polar graph $A_{2D-1}(2)$ & $D$ & see \eqref{classical bi} & see \eqref{classical ci} \\
  \hline
\end{tabular}
\end{adjustbox}
\end{center}
\hfill \break
\begin{center}
\begin{adjustbox}{width=1\textwidth}
\begin{tabular}{|c|c|c|c|}
  \hline
  $t$ from Theorem& Minimal& Classical& Reference\\
  \ref{main theorem}(\ref{Condition 3}) & eigenvalue&parameters&\\
  \hline
  $1$ & $-2$ & $(2,-2,-3,-4)$ & \cite[p.~30~(Examples)]{BCN}\\
  $2$ & $-3$ & $(2,-2,-4,-6)$ & \cite[p.~30~(Examples)]{BCN}\\
  $8$ & $-9$ & $(3,-2,-2,6)$ &  \cite[p.~427]{BCN}\\
  $11$ & $-12$ & $(3,-2,-3,8)$ & \cite[p.~427]{BCN} \\
  $14$ & $-15$ & $(3,-2,-4,10)$ & \cite[p.~428]{BCN}\\
  $3[^{D}_{1}]^{2}-2[^{D}_{1}]-1$ &$2[^{D}_{1}]-3[^{D}_{1}]^{2}$ & $(D,-2,-6,6[^{D}_{1}]-4)$ & \cite[Thm.~9.4.3]{BCN}\\
  \hline
\end{tabular}
\end{adjustbox}
\end{center}
\end{rem}
\hfill \break
\subsection*{Acknowledgements}
\noindent Mojtaba Jazaeri is indebted to Paul Terwilliger for his supervision in obtaining and editing the results of this paper.

\end{document}